\documentclass[11pt,leqno]{amsart}
\usepackage{amscd}
\usepackage{amssymb}
\usepackage{amsfonts}
\usepackage{latexsym}
\usepackage{verbatim}
\usepackage{amsthm}
\numberwithin{equation}{section}

\theoremstyle{plain} \newtheorem{theorem}{Theorem}[section]  \newtheorem{lemma}[theorem]{Lemma} \newtheorem{prop}[theorem]{Proposition}  \newtheorem{rem}[theorem]{Remark}  \newcommand{\ov}[1]{\overline{#1}}       

\begin{document}

\title{Complex submanifolds of almost complex Euclidean spaces}
\author{Antonio J. Di Scala and Luigi Vezzoni} \date{\today} \address{Dipartimento di Matematica \\ Politecnico di Torino \\ Corso Duca degli Abruzzi 24, 10129 Torino, Italy} \email{antonio.discala@polito.it} \address{Dipartimento di Matematica \\ Universit\`a di Torino\\ Via Carlo Alberta 10 \\ 10123 Torino\\ Italy} \email{luigi.vezzoni@unito.it}


\thanks{This work was supported by the Project M.I.U.R. ``Riemannian Metrics and  Differentiable Manifolds'' and by G.N.S.A.G.A. of I.N.d.A.M.}

\begin{abstract} We prove that a compact Riemann surface can be realized as a pseudo-holomorphic curve of $(\mathbb{R}^4,J)$, for some almost complex structure $J$ if and only if it is an elliptic curve. Furthermore we show that any (almost) complex $2n$-torus can be holomorphically embedded in $(\mathbb{R}^{4n},J)$ for a suitable almost complex structure $J$. This allows us to embed any compact Riemann surface in some almost complex Euclidean space and to show many explicit examples of almost complex structure in $\mathbb{R}^{2n}$ which can not be tamed by any symplectic form.

\end{abstract} \maketitle 

\newcommand\C{{\mathbb C}} \newcommand\R{{\mathbb R}} \newcommand\Z{{\mathbb Z}} \newcommand\T{{\mathbb T}} \newcommand\GL{{\rm GL}} \newcommand\SL{{\rm SL}} \newcommand\SO{{\rm SO}} \newcommand\Sp{{\rm Sp}} \newcommand\U{{\rm U}} \newcommand\SU{{\rm SU}} \newcommand{\Gdue}{{\rm G}_2} \newcommand\re{\,{\rm Re}\,} \newcommand\im{\,{\rm Im}\,} \newcommand\id{\,{\rm id}\,} \newcommand\tr{\,{\rm tr}\,} \renewcommand\span{\,{\rm span}\,} \newcommand\Ann{\,{\rm Ann}\,} \newcommand\Hol{{\rm Hol}} \newcommand\Ric{{\rm Ric}} \newcommand\nc{\widetilde{\nabla}} \renewcommand\d{{\partial}} \newcommand\dbar{{\bar{\partial}}} \newcommand\s{{\sigma}} \newcommand\sd{{\bigstar_2}} \newcommand\K{\mathbb{K}} \renewcommand\P{\mathbb{P}} \newcommand\D{\mathbb{D}} \newcommand\al{\alpha} \newcommand\f{{\varphi}} \newcommand\g{{\frak{g}}} \renewcommand\k{{\kappa}} \renewcommand\l{{\lambda}} \newcommand\m{{\mu}} \renewcommand\O{{\Omega}} \renewcommand\t{{\theta}} \newcommand\ebar{{\bar{\varepsilon}}} \newcommand{\dis}{\di^{\circledast}} \newcommand{\fo}{\mathcal{F}} \newcommand{\op}{\overline{\parzial}} \newcommand{\wn}{\widetilde{\nabla}} \newcommand{\I}{\operatorname{i}} \newcommand{\z}{\overline{z}} \newcommand{\wi}{\overline{w}}

\section{Introduction} An almost complex structure is a natural generalization of a complex structure to the non-holomorphic case. Namely, an almost complex structure on a $2n$-dimensional smooth manifold $M$ is a tensor $J\in {\rm End}(TM)$ such that $J^2=-{\rm I}$. An almost complex structure is called \emph{integrable} if it is induced by a holomorphic atlas. In view of the celebrated Newlander-Niremberg theorem, an almost complex structure is integrable if and only if the associated Nijenhuis tensor
$$ N(X,Y):=[JX,JY]-J[JX,Y]-J[X,JY]-[X,Y] $$ vanishes (see \cite{NeNi57}). In dimension $2$ any almost complex structure is integrable, while in higher dimension this is far from true. A smooth map $f\colon N\to M$ between two almost complex manifolds $(N,J')$, $(M,J)$ is called {\em pseudo-holomorphic} if $f_*J=J'f_*$. When the map $f$ is an immersion, $(N,J)$ is said to be an \emph{almost complex submanifold} of $(M,J)$. If $N$ is a compact Riemann surface, then the triple $(N,J',f)$ is called a \emph{pseudo-holomorphic curve}. A local existence result for pseudo-holomorphic curves appeared in \cite{NiWo63}.
The concept of pseudo-holomorphic curves was also implicit in early work of
Gray \cite{Gra65}. Such curves were studied by Bryant in \cite{Bry82} and are
related with the study of harmonic maps and minimal surfaces (see \cite{EeSa86}). The idea of using pseudo-holomorphic curves to study almost complex and symplectic manifolds is due to Gromov. In the celebrated paper \cite{Gro85} Gromov used pseudo-holomorphic curves to introduce new invariants of symplectic manifolds.
Subsequently such curves were taken into account by many authors (see e.g. \cite{McSa04,McSa98} and the references therein).\\

In the present paper we study compact pseudo-holomorphic curves embedded in $\R^{2n}$.
If we equip $\R^{2n}$ with a complex structure, then it does not admit any compact complex submanifold (by the maximum principle). Thus, it is a very natural problem to ascertain if it is possible to find compact complex manifolds embedded in $\R^{2n}$ equipped with a non-integrable almost complex structure. In the first part of this paper we prove the following

\begin{theorem}\label{main}
A compact Riemann surface $X$ can be realized as a pseudo-holomorphic curve of  $\R^{4}$ equipped with an almost complex structure
if and only if it is an elliptic curve.
\end{theorem}

Recall that a complex torus $\mathbb{T}^n$ is the quotient of $\C^n$ by a lattice $\Lambda$.
If $n \geq 2$ it may or may not be algebraic (i.e., an abelian variety). We shall consider more generally $\R^{2n}$ equipped with an
almost complex structure invariant by a lattice. We call the quotient an \emph{almost complex torus}.

\begin{theorem}\label{2main}
Any almost complex torus $\mathbb{T}^n=\R^{2n}/\Lambda$ can be pseudo-holomorphically embedded into $(\R^{4n}, J_{\Lambda})$ for a suitable almost complex structure $J_\Lambda$.
\end{theorem}

Theorem \ref{main} can be applied to prove the following

\begin{theorem}\label{genere}
Any compact Riemann surface can be realized as a pseudo-holomorphic curve of some $(\mathbb{R}^{2n},J)$, where $J$ is a suitable almost complex structure.
\end{theorem}

Using the above results, in the last section we observe that the $J_{\Lambda}$'s are
explicit examples of almost complex structures on $\R^{2n}$ which cannot be tamed by any symplectic form.

\section{Preliminaries}

In the present paper all manifolds, maps, \emph{etc}., are taken to be smooth, i.e., $C^{\infty}$.

Let $J$ be an almost complex structure on a $2n$-dimensional smooth manifold $M$. We denote by $TM$ the tangent bundle to $M$ and by
$TM^*$ its dual bundle.
Then $J$ induces the splitting
$$
\Lambda^{1,0} \oplus \overline{\Lambda^{1,0}} = TM^* \otimes \C\,,
$$
$\Lambda^{1,0}$ being the eigenspace relative to $\I$. On the other hand, if an $n$-dimensional subbundle $\Lambda$ of $TM^* \otimes \C$
satisfying
$$
\Lambda\oplus \overline{\Lambda}=TM^* \otimes \C
$$
is given, then $\Lambda$ determines a canonical almost complex structure $J$ on $M$. Furthermore an almost complex structure allows us to split
the bundle $\Lambda^r\otimes \C$ of the complex $r$-form on $M$ in
\begin{equation}\label{dec}
\Lambda^{r}\otimes \C=\bigoplus_{p+q=r}\Lambda^{p,q}\,.
\end{equation}
The de Rham differential operator splits accordingly as $d=\overline{A}+\overline{\partial}+\partial+A$ with respect to \eqref{dec}. In view of the Newlander-Nirenberg theorem
$J$ is integrable if and only if $A=0$ or, equivalently, if and only if $\overline{\partial}\,^2=0$. We refer the reader to \cite{Sa02} for more details about Hermitian geometry.

If $M$ is a parallelizable  manifold, then the choice of an almost complex structure on $M$ is equivalent to the choice of a
global complex coframe $\{\alpha_1,\dots,\alpha_n\}$ satisfying
$$
\Lambda\oplus\overline{\Lambda}=TM\otimes\C\,, \qquad \Lambda=\span_\C\{\alpha_1,\dots,\alpha_n\}.
$$
The parallelizable case obviously includes  $\R^{2n}$ and the complex tori.

Let be $\Omega$ an open subset of $\R^{2n}$ and let $J$ be an almost complex structure on $\Omega$ determined by a complex coframe
$\{ \alpha_1 , \alpha_2, \cdots, \alpha_n\}$.
Fix a smooth map $f: \Omega \subset \R^{2n} \rightarrow \C$ and denote by $\mathcal{Z}_f :=f^{-1}(0)$ the zero set of $f$.
We are interested in understanding when $\mathcal{Z}_f$ is an almost complex hypersurface of $(\Omega,J)$. We have

\begin{lemma} \label{criterio} Assume that $ df \wedge \overline{df} |_{\mathcal{Z}_f} \neq 0$. Then
the following facts are equivalent:

\begin{itemize}

\item[(i)] $\mathcal{Z}_f$ is an almost complex hypersurface of $(\Omega,J)$, \,
\item[(ii)] $\span_{\C} \{ df, \overline{df} \}_{|\mathcal{Z}_f}$ is $J$-invariant,
\item[(iii)] $df \wedge \overline{df}_{|\mathcal{Z}_f} \in \Lambda^{1,1}$,
\item[(iv)] $df \wedge \overline{df} \wedge \alpha_1 \wedge \alpha_2 \wedge \cdots \wedge \alpha_{n|\mathcal{Z}_f} = 0\,,$
\item[(v)] $\overline{\partial{f}} \wedge \overline{\partial} f_{|\mathcal{Z}_f} = 0\,.$

\end{itemize}
\end{lemma}

\begin{proof}
In order to prove the equivalence $({\rm i}) \Leftrightarrow ({\rm ii})$ we  write $df$ as $df = e^1 + \I e^2$ for
some $e^1,e^2$ in $T{\R^{2n}}^*$. Now it is enough to notice that $\ker(e^1) \cap \ker(e^2) = T_p\mathcal{Z}_f \subset \R^{2n}$ is a $J$-invariant space
if and only if it is possible to complete $\{e^1,e^2\}$ to a basis $\{e^1,e^2,e^3,\cdots,e^{2n}\}$ such that both the spaces
${\rm span}_{\R} \{e^1,e^2 \}$, ${\rm span}_\R\{e^3,\cdots,e^{2n}\}$ are $J$-invariant.

In order to prove  $({\rm ii}) \Rightarrow ({\rm iii})$ we observe that condition (ii) implies that we can assume
$
{\rm span}_{\C} \{ df, \overline{df} \}_{|\mathcal{Z}_f} = \span_{\C} \{ \beta, J\beta \}$ for some real $1$-form $\beta$ satisfying
$ \beta \wedge J \beta =  df \wedge \overline{df}_{|\mathcal{Z}_f}\,.
$
Furthermore
$$
J(\beta \wedge J \beta ) = J \beta \wedge J^2 \beta = \beta \wedge J\beta \,,
$$
which readily implies that
$df \wedge \overline{df}_{|\mathcal{Z}_f} $ is of type $(1,1)$.\\
In order to prove that $({\rm iii}) \Rightarrow ({\rm ii})$ we assume by contradiction that
$J (df) \notin \span_{\C} \{ df, \overline{df} \}_{|\mathcal{Z}_f}$; this implies that
$J (df) \wedge df \wedge \overline{df} \neq 0$. Then $J(J (df) \wedge df \wedge \overline{df}) \neq 0 $, but $(-(df) \wedge J(df) \wedge \overline{Jdf}) = (-(df) \wedge (df) \wedge \overline{df}) = 0$, which is a contradiction.

The equivalence (iii)$\Leftrightarrow$(iv) is obvious.

For the statement (ii)$\Leftrightarrow$(v) we note that the identity $df =\partial{f} + \overline{\partial}f $ implies $\overline{df} = \overline{\partial{f}} + \overline{\overline{\partial}f} $. Hence $ (df \wedge \overline{df})^{0,2} = \overline{\partial}f \wedge \overline{\partial{f}} $, and we are done.
\end{proof}

\section{Proof of the main results}

In this section we prove Theorems \ref{main} and \ref{2main}.
We start by proving a lemma which is a consequence of the adjunction formula \cite[p.139]{McSa98} and the following Theorem due to Whitney:

\begin{theorem}[\cite{Hir76}, page 138] Let $M\subset \R^{2n}$ be a compact oriented $n$-dimensional submanifold. Then $M$ has a non-vanishing normal vector field.
\end{theorem}

\begin{lemma} Let $J$ be an almost complex structure on $\R^4$ and let $i\colon X\hookrightarrow \R^4$ be a pseudo-holomorphic curve of $(\R^4,J)$. Then $X$ is a torus, i.e. $X$ has genus one. \end{lemma}

\begin{proof}[First proof] Let $i\colon X\hookrightarrow \R^4$ be an immersion and denote by $N(X)$ the vector bundle normal to $X$. Let $i^*(T\R^4) = i^*(\R^4 \times \R^4)$ denote the trivial vector bundle on $X$. Then $i^*(\R^8)=TX\oplus N(X)$. By hypothesis $TX$ and $N(X)$ are complex rank one bundles on $X$ and by the above theorem of Whitney $N(X)$ is trivial. Hence, since $\dim_{\C}(N(X))=1$, if $c_1(\cdot)$ denotes the first Chern class, then one has $$ 0=c_1(i^*(\R^8))=c_1(TX\oplus N(X))=c_1(TX)\oplus c_1(N(X))=c_1(TX)\,. $$ Therefore $TX$ has to be trivial and the claim follows. \end{proof}

\begin{proof}[Second proof] Let $J$ be an almost complex structure on $\R^4$ and let $i\colon X\hookrightarrow \R^4$ be a pseudo holomorphic curve of $(\R^4,J)$. Fix an arbitrary $J$-Hermitian metric $g$ on $\R^4$ and let $N(X)$ be the normal bundle to $X$ with respect to $g$. Denote by $\{J_1,J_2,J_3\}$ the quaternionic complex structure of $\R^4$ $$ J_1= \begin{pmatrix} 0& 0 & -1 &0 \\ 0& 0 &  0 &-1\\ 1& 0 &  0 &0\\ 0& 1 &  0 &0 \end{pmatrix}\,,\quad J_2= \begin{pmatrix} 0& 1 &0 &0 \\ -1& 0 &  0 &0\\ 0& 0 &  0 &1\\ 0& 0 &  -1 &0 \end{pmatrix}\,,\quad J_3:=J_1J_2\,. $$ There exists a smooth map $A\colon \R^4\to {\rm GL}(4,\R)$ such that $J_0=A^{-1}JA$.  By the result of Whitney (see e.g. \cite[p.138]{Hir76}) we have that $N(X)$ has a nowhere vanishing global section $V$. Then $W:=AJ_2A^{-1}V$ is a nowhere vanishing tangent vector to $X$. \end{proof}

Using the argument of the second proof one can prove the following more general result:

\begin{prop} Let $J$ be an almost complex structure on $\R^{4n}$ and let $(M^{2n},J_M)$ be an almost complex submanifold of $(\R^{4n},J)$. Then $M$ admits a nowhere vanishing tangent vector field. Thus the Euler characteristic $\chi(M)$ vanishes. \end{prop}

Hence the first part of theorem \ref{main} is proved. Now we show the proof of Theorem \ref{2main}, which in particular implies the second part of Theorem \ref{main}.

\begin{proof}[Proof of Theorem $\ref{2main}$]
First of all we consider the case of the standard elliptic curve $\T=\C/\Z^2$:

Let $J$ be the almost complex structure on $\R^4$ determined by the following complex frame
$$
\begin{aligned} & \alpha_1:=dz -\I zw\,d\overline{w}\,, \\ & \alpha_2 := dw+\I zw \,d\overline{z}\,, \end{aligned}
$$
where $z,w$ are the standard complex coordinates on $\R^4$
and let $f\colon \C^2\to \C$ be the map $f(z,w)=(|z|^2 - 1)+ i (|w|^2 - 1)$. Then $f^{-1}(0)$ defines a smooth embedding of the torus in $\R^4$.
We have
$$
\begin{aligned} & df \wedge \overline{df} \wedge \alpha_1 \wedge \alpha_2 = 2\I d |w|^2 \wedge d|z|^2 \wedge (dz \wedge dw +
(zw)^2 d\wi \wedge d\z) =\\ &2\I d |w|^2 \wedge d|z|^2 \wedge dz \wedge dw +2\I d |w|^2 \wedge d|z|^2 \wedge (zw)^2 d\wi \wedge d\z=\\
&2\I w d\wi \wedge z d\z \wedge dz \wedge dw+ 2\I  \wi dw \wedge \z dz \wedge (zw)^2 d\wi \wedge d\z =\\ & 2\I w z ( d\wi\wedge
 d\z\wedge dz\wedge dw + |zw|^2 dw\wedge dz\wedge d\wi\wedge d\z)=\\ &2 \I z w (1 - |zw|^2)d\wi\wedge d\z\wedge dz\wedge dw \,.
\end{aligned}
 $$
Hence
$$
df \wedge \overline{df} \wedge \alpha_1\wedge \alpha_2 |_{f^{-1}(0)} = 0
$$
and in view of Lemma \ref{criterio} $f^{-1}(0)$ is an almost complex submanifold of $(\R^4,J)$.

Now we consider the general case:

Let $(\T^{2n}:=\R^{4n}/\Z^{4n},J_0)$ be the $2n$-dimensional torus equipped with the standard complex structure.
In this case, the space of the $(1,0)$-forms is generated by $\{dz_1,\dots,dz_n\}$.

Any almost complex structure $J'$ on $\T^{n}$ can be described by a coframe $\{\zeta_1,\dots,\zeta_n\}$ of the form
$$
\zeta_i=\sum_{j=1}^n\tau_{ij}\,dz_j+\sum_{\ov{j}=1}^n\tau_{i\ov{j}}\,dz_{\ov{j}}\,,\mbox{ for }i=1,\dots, n\,,
$$
where the $\tau_{ij}$'s and the $\tau_{i\ov{j}}$'s are complex maps on $\T^{2n}$. Such an almost complex structure corresponds to
a  lattice $\Lambda$ in $\R^{2n}$.
In view of the first step of this proof, the standard complex torus $(\T^{2n}:=\R^{4n}/\Z^{4n},J_0)$
can be holomorphically embedded in $(\R^{4n},J)$, for a suitable $J$.
Let $\{\alpha_1,\dots,\alpha_n,\beta_1,\dots,\beta_n\}$ be a $(1,0)$-coframe on $(\R^{4n},J)$ satisfying
\begin{equation}\label{base}
i^*(\alpha_i)=i^*(\beta_i)=dz_i\,,\quad i=1,\dots,n\,.
\end{equation}
The maps  $\{\tau_{ij}\}\,,\{\tau_{i\ov{j}}\}$ can be regarded as periodic maps on $\R^{2n}$ and they can be extended on $\R^{4n}$. We can assume
that conditions \eqref{base} still hold on $\R^{4n}$. Now set
$$
\begin{aligned}
\alpha_i^{\tau}:=\sum_{k=1}^n\,\tau_{ik}\alpha_k+\sum_{k=1}^n\tau_{i\ov{k}}\,\beta_{\ov{k}}\,,\\
\beta_j^{\tau}:=\sum_{k=1}^n\,\tau_{ik}\beta_k+\sum_{k=1}^n\tau_{i\ov{k}}\,\alpha_{\ov{k}}\,,
\end{aligned}
$$
for $i,j=1\dots,n$. The frame $\{\alpha_1^{\tau},\dots,\alpha_n^{\tau},\beta_{1}^{\tau},\dots,\beta_n^{\tau}\}$
induces an almost complex structure $J_{\Lambda}$ on $\R^{4n}$. A direct computation gives
$$
i^*(\alpha^{\tau}_i)=i^*(\beta^{\tau}_i)=\zeta_i\,,\mbox{ for } i=1,\dots,n\,.
$$
Hence $i\colon (\T^{2n},J')\hookrightarrow (\R^{4n},J_{\Lambda})$ is a pseudo-holomorphic embedding.
\end{proof}
Using the well-known Abel-Jacobi map one can embed any compact Riemann surface of positive genus in its Jacobian variety.
Hence in view of above theorem any compact Riemann surface of positive genus can be embedded in some $(\R^{2N},J)$.
So, in order to prove Theorem \ref{genere}, we have to show that $\mathbb{CP}^1$ can be embedded in some $(\R^{2n},J)$ for a suitable $J$.
We have the following
\begin{lemma}
The Riemann sphere $\mathbb{CP}^1$ can be pseudo-holomorphically embedded in $(\R^{6},J)$, for a suitable $J$.
\end{lemma}
\begin{proof}
The sphere $S^6$ has a canonic almost complex structure $J$ induced by the algebra of the  octonians $\mathbb{O}$ (see e.g., \cite{Ca58,Bry82}). Moreover, if we intersect $S^6$ with a suitable
$3$-dimensional subspace of $\rm{Im}\,\mathbb{O}$, then we obtain a pseudo-holomorphic embedding of $\mathbb{CP}^1$ in $(S^6,J)$ (see \cite{Bry82}, again). Hence the stereographic projection $\pi\colon S^6-\{\mbox{a point}\}\to \R^6$ induces a pseudo-holomorphic embedding of $\mathbb{CP}^1$ to $\mathbb{R}^6$ equipped with the almost complex structure induced by $J$ and $\pi$.
\end{proof}
\section{Almost complex structures tamed by symplectic forms}

We recall that an almost complex structure $J$ on a manifold $M$ is said to be {\em tamed} by a symplectic form $\omega$ if $$ \omega_x(v,J_xv) > 0 $$ for any $x\in M$ and any $v\in T_xM$, $v\neq 0$.
Let us remark the following known fact

\begin{prop}
Let $(M,\omega)$ be a symplectic manifold such that the second $($de Rham$)$ cohomology group $\,{\rm H}^{2}(M)$ is trivial and let $J$ be an almost complex structure on $M$ tamed by $\omega$. Then $(M,J)$ does not admit any compact almost complex submanifold.
\end{prop}

\begin{proof}
Since the second cohomology group is trivial the symplectic form $\omega$ is exact, namely $\omega = d\theta$. If $N \hookrightarrow M$ is almost complex then $(N,\omega = d\theta)$ is exact symplectic and so Stokes Theorem implies that $N$ cannot be compact
\end{proof}

Here is an application:

\begin{prop}
Let $J$ be an almost complex structure on $\R^{2n}$ tamed by a symplectic form. Then $(\R^{2n},J)$
does not admit any compact almost complex submanifold. In particular, the almost complex structures $J_{\Lambda}$'s described in section $2$
cannot be tamed by any symplectic form.
\end{prop}

\begin{rem}{\em
For other examples of almost complex structures which can not be calibrated by symplectic forms see \cite{MiTo00,To02}}.
\end{rem}
\vspace{0.5cm}
\begin{center}
{\bf Acknowledgment}
\end{center}

We thank Simon Salamon, Paolo de Bartolomeis and Diego Matessi for useful comments and remarks.


\begin{thebibliography}{12345678}
\bibitem[Bry82]{Bry82}
{\sc Bryant, R. L.:} {\it Submanifolds and special structures on the octonians.} J. Differential Geom.  17  (1982), no. 2, 185--232.




\bibitem[Ca58]{Ca58}
{\sc Calabi, E.:} {\it Construction and properties of some $6$-dimensional almost complex manifolds.}
Trans. Amer. Math. Soc.  87  (1958) 407--438.

\bibitem[EeSa86]{EeSa86}
{\sc Eells, J.; Salamon, S.:} {\it Twistorial construction of harmonic maps of surfaces into four-manifolds.}
Ann. Scuola Norm. Sup. Pisa Cl. Sci. (4)  12  (1985),  no. 4, 589--640 (1986).

\bibitem[Gra65]{Gra65}
{\sc Gray, A.:} {\it Minimal varieties and almost Hermitian submanifolds.}
Michigan Math. J.  12  1965 273--287.

\bibitem[Gro85]{Gro85}
{\sc Gromov, M.:} {\it Pseudoholomorphic curves in symplectic manifolds.}
Invent. Math. 82 (1985), no. 2, 307--347.

\bibitem[Hir76]{Hir76}
{\sc Hirsch, Morris W.:} {\it Differential topology.}
Graduate Texts in Mathematics, No. 33. Springer-Verlag, New York-Heidelberg, 1976.

\bibitem[McSa98]{McSa98}
{\sc McDuff, D.; Salamon, D.:} {\it Introduction to Symplectic Topology.}
Second Edition, Clarendon Press-Oxford, (1998).

\bibitem[McSa04]{McSa04}
{\sc McDuff, D.; Salamon, D.:} {\it $J$-holomorphic curves and symplectic topology.} American Mathematical Society Colloquium Publications, 52. American Mathematical Society, Providence, (2004).

\bibitem[MiTo00]{MiTo00}
{\sc Migliorini M.; Tomassini A.:} {\it Local calibrations of almost complex structures.}  Forum Math.  12  (2000),  no. 6, 723--730.

\bibitem[NeNi57]{NeNi57}
{\sc Newlander, A.; Nirenberg, L.:} {\it Complex analytic coordinates in almost complex manifolds.}
Ann. of Math. (2) 65 (1957), 391--404.

\bibitem[NiWo63]{NiWo63}
{\sc Nijenhuis, A.; Woolf, W.B.:} {\it Some integration problems in almost-complex and complex manifolds.}
Ann. of Math. (2) 77 1963 424--489.

\bibitem[Sa02]{Sa02}
{\sc Salamon, S. M.:} {\it Hermitian geometry.  Invitations to geometry and topology.}  233--291, Oxf. Grad. Texts Math., 7, (2002).

\bibitem[To02]{To02}
{\sc Tomassini A.:} {\it Some examples of non calibrable almost complex structures.}  Forum Math.  14  (2002),  no. 6, 869--876.

\end{thebibliography}
\end{document}